\documentclass[11pt,reqno]{amsart}
\usepackage{setspace}
\usepackage{amsmath,amsthm,amssymb,amsfonts,hyperref,mathtools,graphicx}
\usepackage{graphicx}
\usepackage{xcolor}
\usepackage{tikz}
\usepackage[marginparwidth=5em]{geometry}
\mathtoolsset{showonlyrefs}

\usepackage[normalem]{ulem}

\usepackage{cancel}
\newtheorem{theoremA}{Theorem}
\newtheorem{lemmaA}[theoremA]{Lemma}
\newtheorem{propositionA}[theoremA]{Proposition}

\newtheorem{theorem}{Theorem}[section]
\newtheorem{corollary}[theorem]{Corollary}
\newtheorem{lemma}[theorem]{Lemma}
\newtheorem{proposition}[theorem]{Proposition}

\newtheorem{remark}[theorem]{Remark}
\newtheorem{example}[theorem]{Example}

\newcommand{\mel}{\MoveEqLeft}

\allowdisplaybreaks[3]

\numberwithin{equation}{section}

\newcommand\dx{\,dx}

\newcommand\dt{\,dt}

\newcommand{\kabs}[1]{\ensuremath{\vert#1\vert}}
\newcommand{\babs}[1]{\ensuremath{\big\vert#1\big\vert}}
\newcommand{\Babs}[1]{\ensuremath{\Big\vert#1\Big\vert}}

\DeclareMathOperator{\Div}{div}

\def\R{\mathbb R}

\def\eps{\varepsilon}
\def\eu{{\rm Eucl}}

\def\part#1#2{\par\noindent{\underline{\it Part~#1.}}\emph{ #2}\\}
\def\XXint#1#2#3{{\setbox0=\hbox{$#1{#2#3}{\int}$} \vcenter{\vspace{-1pt}\hbox{$#2#3$}}\kern-.5\wd0}}

\newcounter{mt}

\title[Optimal regularity of isoperimetric sets with H\"older densities]{Optimal regularity of isoperimetric sets with H\"older densities}

\author{Lisa Beck}
\address[LB]{Institut f\"ur Mathematik\\
Universit\"at Augsburg\\
 	Universit\"atsstra\ss e 12a\\
 	86159 Augsburg\\
 	Germany}
 	\email{lisa.beck@math.uni-augsburg.de}
\author{Eleonora Cinti}
\address[EC]{Dipartimento di Matematica\\
Universit\`a di Bologna\\
Piazza di Porta San Donato 5\\
Bologna\\
Italy}
\email{eleonora.cinti5@unibo.it}
\author{Christian Seis}
\address[CS]{Institut f\"ur Analysis und Numerik\\
Westf\"alische Wilhelms-Universit\"at M\"unster\\
 Einsteinstra\ss e 62\\
 48149 M\"unster\\
 Germany}
 \email{seis@wwu.de}

\begin{document}
\begin{abstract}
We establish a regularity result for optimal sets of the isoperimetric problem with double density  under  mild {($\alpha$-)}H\"older regularity assumptions on the density functions. Our main Theorem improves some previous results  and allows to reach in any dimension the regularity class~$C^{1,\frac{\alpha}{2-\alpha}}$. {This class is indeed the optimal one for local minimizers of variational functionals with an integrand that depends $\alpha$-Hölder continuous on the minimizer itself, and as such can (the boundary of) the isoperimetric set be locally written (with additional constraint).
}
\end{abstract}

\maketitle

\section{Introduction}

In this paper we are concerned with the regularity of isoperimetric sets in $\R^n$ with densities, for arbitrary dimensions $n \geq 2$ and with an emphasis on the situation of  very low regularity assumptions on the density functions. Isoperimetric sets are defined as solutions of the isoperimetric problem in $\R^{n}$ with densities, which can be formulated as follows: 
Given two lower semi-continuous functions $f,\,h \colon \R^{n} \to (0,+\infty)$, the so-called \emph{densities}, and an arbitrary measurable set $E \subset \R^{n}$, we introduce its (weighted) volume $V_f(E)$ via
\begin{equation*}
 V_f(E) \coloneqq \int_E f(x) \dx 
\end{equation*}
and its (weighted) perimeter $P_h(E)$ via 
\begin{equation*}
 P_h(E) \coloneqq \int_{\partial^* E} h(x) \, d\mathcal{H}^{n-1}(x),
\end{equation*}
whenever~$E$ is of locally finite perimeter (and with $\partial^* E$ denoting the reduced boundary of~$E$), while we set $P_h(E) \coloneqq {+}\infty$ otherwise. For some positive number~$m$, we then look for a set of minimal weighted perimeter among all sets of fixed weighted volume~$m$, i.e., we look for minimizers of
\begin{equation*}
 \inf \big\{ P_h(E) \colon E \subset \R^n \text{ with } V_f (E) = m \big\}.
\end{equation*}

The classical isoperimetric problem dating back to ancient Greece corresponds to constant density functions $h=f \equiv 1$, for which the weighted volume and perimeter reduce to the Euclidean volume~$V_\eu$ and Euclidean perimeter~$P_\eu$, and in this case it is well known that the isoperimetric sets relative to a constant~$m$ are precisely all balls of radius $R$ such that the equality
\begin{equation*}
V_\eu(B_R) = \frac{\pi^{n/2}}{\Gamma(1+n/2)} R^n = m
\end{equation*}
is satisfied. 


In the interesting case of non-constant densities, existence of isoperimetric sets is still guaranteed under quite general assumptions on the density functions (see \cite{CP1,DFP,MP,PS1}). These sets are in general not unique, even not in the equivalence class of spatial translates. (Notice, however, that in certain geometries, uniqueness of the isoperimetric set is obtained only \emph{thanks} to the weight.) 
In general, if the weight functions are not globally bounded and continuous, the isoperimetric sets are not necessarily bounded. The complementary, however, is true as proved in \cite[Theorem 1.1]{CP1} and \cite[Theorem B]{PS1}.


In this paper, we focus on the regularity of the isoperimetric sets, or more precisely, on the regularity of their boundaries. (Optimal) regularity has been investigated for many years in dependency on the regularity of the density functions. 

A classical (and optimal) result in this regard for the case of a single density (i.e., $f=h$), under the assumption of quite high regularity, which in particular allows to take advantage of the Euler--Lagrange formulation of the problem, is the following statement (\cite[Proposition~3.5 and Corollary~3.8]{Morgan2003}).

\begin{theoremA}[\cite{Morgan2003}]
\label{prop_densities_high_reg}
Let $f=h$ be of class $C^{k,\alpha}(\R^n,\R^+)$ for some $k \geq 1$ and $\alpha \in (0,1]$. Then the boundary of any isoperimetric set is of class $C^{k+1,\alpha}$, except for a singular set of Hausdorff dimension at most $n-8$. 
\end{theoremA}

When the densities are assumed to be just H\"older functions, one cannot even write down the associated Euler--Lagrange equation.
This case of poorly regular densities was addressed only recently. 
More specifically, in  Theorem~5.7 of \cite{CP1} a first regularity result for the case of a single density $f=h$, which is H\"older continuous of order $\alpha$, was established in any dimension (with a H\"older exponent depending on $\alpha$ and on the dimension). A second result   improved the above regularity result in the $2$-dimensional case, see Theorem~A in \cite{CP2}. Here is the precise combined statement.


\begin{theoremA}[\cite{CP1,CP2}]\label{CP}
\label{prop_densities_low_reg1}
Let $f=h$ be  of class $C^{0,\alpha}(\R^n,\R^+)$ for some $\alpha \in (0,1]$. Then, if $E$ is an isoperimetric set, we have that $\partial E=\partial^*E$ up to $\mathcal H^{n-1}$-negligible sets, and  $\partial^*E \in C^{1, \alpha/(2n(1-\alpha) + 2\alpha)}$. If $n=2$, we have that $\partial^*E \in C^{1,\alpha/(3-2\alpha)}$.
\end{theoremA}


More recently, in Theorem~C of \cite{PS} the regularity result in dimension $n$ was generalized to the case of two different densities. As it is clear from the proof of this result, the H\"older regularity of an optimal set only depends on the H\"older regularity of the density $h$ (weighting the perimeter), while no regularity is needed on $f$. 

\begin{theoremA}[\cite{PS}]\label{PS}
	\label{prop_densities_low_reg2}
	Let $h$ be a density of class $C^{0,\alpha}(\R^n,\R^+)$ for some $\alpha \in (0,1]$ and $f$ be a locally bounded function. Then, if $E$ is an isoperimetric set, we have that $\partial E=\partial^*E$ up to $\mathcal H^{n-1}$-negligible sets, and  $\partial^*E \in C^{1, \alpha/(2n(1-\alpha) + 2\alpha)}$.
\end{theoremA}

We will later build up on these regularity results and we will thus refer to $C^{1,\sigma}$  with  $\sigma=\alpha/(2n(1-\alpha) + 2\alpha)$ as the \emph{initial  regularity}.

The proof of the regularity result in any dimension (for both the case of single and double density) consists in showing that if $E$ is an isoperimetric set with an $\alpha$-H\"older density $h$, then it is an $\omega$-minimal set (or almost-minimal set) for a certain modulus of continuity $\omega(r)=r^{2\sigma}$. Hence, standard regularity theory for $\omega$-minimal sets applies and allows to obtain $C^{1,\sigma}$ regularity of $\partial E$.

We observe that using this approach, which relies on $\omega$-minimality, the order $\sigma$ that one can reach tends to $1/2$ when $\alpha\rightarrow 1$. In the 2-dimensional result of Theorem \ref{prop_densities_low_reg1}, the exponent $\frac{\alpha}{3-2\alpha}$ is still not-optimal but tends to $1$ for $\alpha \rightarrow 1$.

The crucial ingredient in the proof of both Theorems \ref{prop_densities_low_reg1} and \ref{prop_densities_low_reg2} is the so-called $\eps-\eps^\beta$ property, first established in  \cite[Theorem B]{CP1} for the case of a single density, and then generalized to the case of double density in \cite{PS}. Roughly speaking this property says that it is possible to modify a set $E$ by changing its volume of an amount $\eps$ and increasing its perimeter of an amount proportional to at most  $\eps^\beta$. In the case of a Lipschitz density the exponent $\beta$ can be chosen to be $1$, while for a  H\"older density it must be chosen depending on the H\"older regularity of the density. In the case of double density, as mentioned before, only the density on the perimeter is needed to be H\"older continuous. 

The aim of the present paper is to improve these regularity results in the setting of  H\"older continuous densities and in general dimensions. More specifically, we prove the following:

\begin{theorem}
\label{thm_main}
Let $h$ be a density of class $C^{0,\alpha}(\R^n,\R^+)$ and $f$ be a density of class $C^{0,\gamma}(\R^n,\R^+)$ for some $\alpha$ and $\gamma \in (0,1)$.  Then the boundary of any isoperimetric set is of class $C^{1,\alpha/(2-\alpha)}$, except for a singular set of Hausdorff dimension at most $n-8$. 
\end{theorem}

{We emphasize that the H\"older exponent $\frac{\alpha}{2-\alpha}$ in our statement is the expected optimal one, see also Example \ref{Ex1} below.} Moreover, it does not depend on the dimension and it indeed improves on  the regularity results established before because
\begin{equation*}
 \tfrac{\alpha}{2 - \alpha} > 
 \begin{cases} \begin{array}{l l}
      \frac{\alpha}{2n(1-\alpha) + 2\alpha} \quad & \text{for } n \geq 2, \\
      \frac{\alpha}{3-2\alpha} \quad & \text{for } n = 2\mbox{ and }f=h, 
 \end{array} \end{cases}  
\end{equation*}
for each $\alpha \in (0,1)$. 
 In particular, also the expected asymptotic behavior $\alpha/(2 - \alpha) \to 1$ as $\alpha \nearrow 1$ is now achieved for all dimensions $n \geq 2$. 
 
 At first glance, the loss in the order of the H\"older semi-norm from $\alpha$ in the differentiable setting of Theorem \ref{prop_densities_high_reg} to $\frac{\alpha}{2-\alpha}$ in the  continuous setting in Theorem \ref{thm_main} seems surprising.  This feature  is, however, well known from classical regularity theory for the minimization of variational functionals of the form $\mathcal F[w]= \int_\Omega F(x,w,Dw)\, dx$ among Sobolev functions in a given Dirichlet class, in the specific situation that merely an $\alpha$-H\"older continuity assumption is imposed on the maps $u \mapsto F(x,u,z)$ and $(x,u) \mapsto D_z F(x,u,z)$ which does not allow for the passage to an Euler--Lagrange equation. In this case, the optimal regularity of minimizers is precisely $C^{1, \alpha/(2 - \alpha)}$, see~\cite{Philips}.

The hypothesis that the volume density~$f$ must be Hölder continuous seems to be an artifact of our method of proof and we actually do not believe that it is a necessary assumption. Indeed, our optimal regularity result is independent of the H\"older exponent $\gamma$, and the suboptimal results from Theorem \ref{prop_densities_low_reg2} hold true also without any continuity requirement. Unfortunately, we do currently not see how to remove this hypothesis.

Finally, the optimal bound on the Hausdorff dimension of the singular set follows from the standard regularity theory for $\omega$-minimal sets, established by Tamanini in \cite{Tam}. Indeed, as commented above, in \cite{CP1} and \cite{PS} it is proved that an isoperimetric set with density is $\omega$-minimal for a certain modulus of continuity $\omega(r)=r^{2\sigma}$, and hence Tamanini's regularity result applies.

We conclude this introduction by discussing the strategy of the proof of our main result.




 Thanks to the already known regularity result of Theorem~\ref{prop_densities_low_reg2}, we can work with a local representation of the reduced boundary of isoperimetric sets (at some given regular point) in terms of the graph of a $C^1$-function. More precisely, in order to study the local regularity of isoperimetric sets, this amounts to considering minimizers~$u$ of the functional
\begin{equation}\label{3}
w \mapsto  \int_{B_R(0)} h(x',w) \big( 1 + \kabs{Dw}^2 \big)^{\frac{1}{2}} \dx' 
\end{equation}
 among all functions $w$ satisfying the constraint
\begin{equation}\label{4}
 \int_{B_R(0)} \int_0^{w(x')} f(x',t) \dt \dx' = m
\end{equation}
for a given constant $m$ and with prescribed boundary values on $\partial B_R(0)$. Notice that here, $B_R(0)$ denotes an open ball in $\R^{n-1}$ that we have centered at the origin for convenience and we may choose $R\le R_0\le 1$. 


Our method for establishing (optimal) H\"older regularity is based on the so-called direct approach from classical regularity theory for minimization problems, see e.g.~\cite{GG82,GG83} for the original theory without integral constraints, and \cite{ACM,Beck} for more recent text books. 
To this end, we define in Section~\ref{Section_comparison} a suitable comparison problem by keeping the original density only for the volume constraint, and by freezing it for the perimeter. Via a combination of the initial regularity result for the minimizer of the comparison problem, suitable estimates on the Lagrange multiplier (coming from the volume constraint),  and classical Schauder theory (applied to the associated Euler--Lagrange equation which does now indeed exist), the minimizer is then shown to have optimal decay estimates. Finally, in Section~\ref{Section_regularity_proof} the decay estimates of this comparison function are then carried over to the minimizer of the original constrained minimization problem, thus completing the proof of Theorem~\ref{thm_main}, via the Campanato characterization of H\"older continuous functions.

We finally want to give an example that shows that the expected optimal H\"older exponent of the boundary of isoperimetric sets in our setting is $\frac{\alpha}{2-\alpha}$, which we find in this paper. 

\begin{example}\label{Ex1}
For simplicity, we consider the two-dimensional problem and an isoperimetric set which locally can be written as the graph of a function~$w$ over the interval~$(0,\ell)$, with $w>0$ in~$(0,\ell)$ and $w(0)=0$. We suppose that the volume density is locally constant, with $f\equiv 1$, and that the {perimeter density is locally of the form $h(x_1,x_2) = H(x_1, |x_2|^{\alpha})$ for some $\alpha \in (0,1)$ and some function~$H$ such that both $H$ and its derivative $\partial_2 H$ with respect to the second variable are bounded from below and from above by positive constants}. Then, according to \eqref{3} and \eqref{4}, $w$ minimizes the functional
	\[
	{w\mapsto \int_0^\ell H(z,|w(z)|^\alpha) \sqrt{1+(w'(z))^2} \, dz,}
	\]
	among all $w$ satisfying the constraint
	\[
	\int_0^\ell w(z)\, dz = m.
	\]
	In this scenario, we may compute the Euler--Lagrange equation inside the interval and find that
	\[
	{\left(H(z,|w|^\alpha) \frac{w'}{\sqrt{1+ (w')^2}} \right)'   + \partial_2 H(z,|w|^\alpha) \alpha w^{\alpha-1}  \sqrt{1+(w')^2}= \lambda} ,
	\]
	where $\lambda \in \R$ is the Lagrange multiplier. In view of the assumption $w(0)=0$ and knowing that $w$ is a $C^{1,\sigma}$ function, it must hold that $w(z) \approx z^{1+\sigma}$ near $z=0$. Plugging this Ansatz into the Euler--Lagrange equation, we see that the first term is of the order $O(z^{\sigma-1})$, while the second one is $O(z^{(\alpha-1)(1+\sigma)})$. The singularity of both terms enforces that both exponents are identical, $\sigma-1 = (\alpha-1)(1+\sigma)$, which yields $\sigma = \frac{\alpha}{2-\alpha}$.
 \end{example}

The paper is organized as follows:
\begin{itemize}
	\item Section 2 is devoted to introduce some notations and preliminaries;
	\item in Section 3, we introduce the comparison problem and prove some crucial preliminary Lemmas that relate, in a quantitative way, the Lagrange multiplier $\lambda$ with the $L^2$-distance between the gradient of the comparison function and the gradient of the solution of our original weighted problem;
	\item Section 4, which is the core of the paper, deals with decay estimates for the comparison function $v$;
	\item finally, in Section 5, we transfer these decay estimates from $v$ to the solution $u$ of our original problem and deduce our regularity result. 
\end{itemize}

\section{Notation and preliminaries}\label{Section:prelim}
We start by introducing some notation.
In the following, we will denote by $x=(x',x_n)$ a point in $\R^n=\R^{n-1}\times \R$.
Given $r>0$ and $x_0'\in \R^{n-1}$, we denote by $B_r(x'_0)$ the ball in $\R^{n-1}$ centered at $x_0'$ and with radius $r$, and we write   $B_r \coloneqq B_r(0)$ for simplicity.

Given a measurable function $w$ defined on $\R^{n-1}$, we denote its mean integral on a certain measurable set $A\subset \R^{n-1}$ by
$$(w)_A \coloneqq \frac{1}{|A|}\int_A w(x')\,dx',$$
where $|A|$ stands for the Lebesgue measure of $A$.
In the particular case in which $A=B_r(x_0')$ and it is clear from the context what is the center $x_0'$, we simply use the abbreviation $(w)_r$ instead of $(w)_{B_r(x_0')}$.

As mentioned above, we will rely on Campanato's characterization of H\"older continuity, which we recall here. See also Section 3.1 in \cite{ACM}.

\begin{propositionA}[\cite{Campanato63}, Teorema~I.2]
\label{lemma_Campanato}
Let $B_R$ be a ball in $\R^{n-1}$, $\beta \in (0,1]$ and $p \in [1,\infty)$. A function $w \in L^1(B_R)$ is \textup{(}up to the choice of a suitable representative\textup{)} H\"older continuous with exponent~$\beta$, i.e.,~$w \in C^{0,\beta}(\overline{B_R})$, if and only if there exists a constant~$C$ such that for each ball $B_\rho(y')$ centered in some point $y' \in B_R$ there holds
\begin{equation*}
\int_{B_R \cap B_\rho(y')} \babs{w - (w)_{B_R \cap B_\rho(y')}}^p \dx' \leq C\rho^{n-1+ p\beta}.
\end{equation*}
\end{propositionA}

The  oscillation on the left-hand side is a monotone function of $\rho$ because for any measurable set $A\subset \R^{n-1}$ the mapping 
\begin{equation}
\label{eq_minimality_mean_value}
c \mapsto \int_A {|w-c|^p} \, dx' \quad { \text{is minimized at } c = (w)_A}.
\end{equation}
The following  iteration lemma is thus well-suited for  oscillations  and it is an elementary though fundamental tool in elliptic regularity theory as it allows to pass from $\rho^{\alpha_2}$ to $r^{\alpha_2}$ and to drop an additive non-decaying term (the one that involves $\eps$)  at the expense of lowering the order of decay. See also Lemma 3.13 in \cite{ACM}.

\begin{lemmaA}[\cite{Giaquinta83}, Lemma~III.2.1]
\label{lemma_iteration}
Assume that $\phi(\rho)$ is a non-negative, real-valued, non-decreasing function defined on the interval $[0,R_0]$ which satisfies
\begin{equation*}
  \phi(r) \leq C_1 \Big[ \Big(\frac{r}{\rho}\Big)^{\alpha_1} + \varepsilon \Big] \phi(\rho) + C_2 \rho^{\alpha_2}
\end{equation*}
for all $r \leq \rho \leq \rho_0$, some non-negative constants $C_1$, $C_2$, and positive exponents $\alpha_1 > \alpha_2$. Then there exists a positive number $\varepsilon_0 = \varepsilon_0(C_1,\alpha_1,\alpha_2)$ such that for $\varepsilon \leq \varepsilon_0$ and all $r \leq \rho \leq \rho_0$ we have
\begin{equation*}
\phi(r) \leq c(C_1,\alpha_1,\alpha_2) \Big[ \Big( \frac{r}{\rho} \Big)^{\alpha_2} \phi(\rho) + C_2 r^{\alpha_2} \Big].
\end{equation*}
\end{lemmaA}

Notice that if the quantity $\phi$ in the lemma is indeed the oscillation, the statement implies H\"older regularity of the order  $(\alpha_2 - n+1)/p$ via Proposition~\ref{lemma_Campanato}.

We now collect at one spot  for the reader's convenience  the H\"older regularity of the density functions:
\begin{equation*}
	\begin{aligned}
		h\in C^{0,\alpha}  &\quad \mbox{regularity of perimeter density},\\
f\in C^{0,\gamma}  &\quad \mbox{regularity of volume density}. 
\end{aligned}
\end{equation*}
Since our final statement in Theorem \ref{thm_main} is independent of $\gamma$, we may without loss of generality suppose that $\gamma $ is small, for instance, 
\begin{equation}
\label{2}
\gamma <\min\bigg\{\frac{\alpha}2,\frac{2(1-\alpha)}{2-\alpha}\bigg\}.
\end{equation}
This choice will slightly simplify our notation later on.


%
%
%

Finally, we conclude this section with a comment on local bounds for the density functions. 
Since we are assuming that the densities 
$f$ and $h$ are positive continuous functions,  they are in particular locally bounded both from above and  away from zero. In particular, for any $T>0$, there exists a constant $M>0$ that can be chosen independently from our localization scale $R\le R_0$ introduced in \eqref{3} and \eqref{4} such that
\begin{equation}
\label{local-bound}
\frac{1}{M}\le f(x',t)\le M\quad\mbox{and}\quad \frac{1}{M}\le h(x',t)\le M \quad\mbox{for any } (x',t)\in B_R\times (-T,T).
\end{equation}
{For a fixed local representation of the reduced boundary as the graph of a $C^1$-function~$u$ minimizing the weighted surface area functional~\eqref{3} under the constraint \eqref{4}, } we will always choose $T$ large enough so that $\|u\|_{L^{\infty}(B_R)} \le T$ uniformly in $R\le R_0$, which allows for choosing $t=u(x')$.
Moreover, as a consequence of the initial regularity statement of Theorem \ref{prop_densities_low_reg2}, the gradient of~{$u$} is locally bounded by a constant $K \geq 1$ that is independent of $R\le R_0$, i.e., 
\begin{equation}\label{8}
	\|D u\|_{L^\infty(B_R)}\le K.
\end{equation}

In the rest of the paper, when we write
$A \lesssim B$, we mean that there exists a constant $c$ which depends only on $n,\,K,\,M,\,[f]_{C^{0,\gamma}}$ and $[h]_{C^{0,\alpha}}$ (hence, in particular, it is independent of~$R$), such that $A \le c B$.

\section{The comparison problem and bounds on the Lagrange multiplier}
\label{Section_comparison}


As outlined in the introduction, we will study the regularity of the boundary of an isoperimetric set with density $E$ by means of classical regularity estimates for the (volume-constrained) minimizer of the weighted surface area functional
\begin{equation*}
w \mapsto  \int_{B_R} h(x',w) a(Dw) \dx' ,
\end{equation*}
where, for notational convenience, we have used the abbreviation $a(z) \coloneqq ( 1 + \kabs{z}^2 )^{\frac{1}{2}}$ and $B_R$ stands, as before, for an $n-1$ dimensional Euclidean ball of radius $R$ (that we  assume to be centered at $0$). The derivation of this functional as the graph representation of the boundary is fairly standard in the context of isoperimetric and  minimal surface problems, we omit the details. In the following discussion, the minimizer will be denoted by $u$.

The proof of our main regularity result is based on  the derivation of regularity estimates for a comparison problem, in which the density $h(x',w)$ in the surface area functional is frozen to a constant (and thus removed) in order to allow for the application of elliptic regularity arguments in the now computable Euler--Lagrange equations.

A further simplification is achieved by modifying the surface function $a(z)$ for large values of $z$. 
In fact, by \eqref{8}, we have that the variational problem remains unchanged if we substitute $a(z)$ by an increasing smooth and strongly convex function $a_K \geq a$ such that
\[
a_K(z) = \begin{cases} \left(1+|z|^2\right)^{\frac12} &\mbox{ if }|z|\le K,\\
c_K(1+|z|^2) &\mbox{ if }|z|\ge 2K,
\end{cases}
\]
for some constant $c_K$. Notice that by strong convexity, we mean that there exists a constant $\mu>0$ such that
\begin{equation}
\label{12}
a_K(z_2) \ge a_K(z_1) + D_z a_K(z_1)\cdot(z_2-z_1) + \frac{\mu}2|z_2-z_1|^2,
\end{equation}
for any $z_1,z_2\in \R^{n-1}$. This statement is equivalent to the ellipticity condition 
\begin{equation}
\label{13}
\xi\cdot D^2_z a_K(z) \xi \ge \mu |\xi|^2,
\end{equation}
for any $\xi\in \R^{n-1}$.

Our \emph{comparison problem} is thus the following: We study the problem
 of minimizing the (Euclidean) perimeter (still in the graph representation) 
\begin{equation}
\label{def_comp_functional}
 w \mapsto \int_{B_R}  a_K(Dw) \dx'
\end{equation}
among all functions $w  $ with $w=u$ on the boundary $\partial B_R$ which satisfy the weighted volume constraint
\begin{equation}
\label{def_comp_constraint}
 \int_{B_R} \int_0^{w(x')} f(x',t) \dt \dx' = \int_{B_R} \int_0^{u(x')} f(x',t) \dt \dx'.
\end{equation}

With standard compactness arguments, we obtain existence of minimizers of the comparison problem.

\begin{lemma}[Euler--Lagrange equations]
In the above setting, a minimizer $v \in u + W^{1,2}_0(B_R)$ to the functional~\eqref{def_comp_functional} under the constraint~\eqref{def_comp_constraint} always exists. Moreover, every minimizer $v$ satisfies an Euler--Lagrange equation with Lagrange multiplier $\lambda \in \R$: for every $\varphi \in W^{1,2}_0(B_R)$, there holds
\begin{equation}
\label{EL_comparison}
 \int_{B_R} D_z a_K(Dv(x')) \cdot D\varphi(x') \dx' + \lambda \int_{B_R} f(x',v) \varphi(x') \dx' = 0.
\end{equation}
\end{lemma}

\begin{proof}Because $a_K(z)$ has quadratic growth, the direct method yields the existence of a minimizer of the comparison problem \eqref{def_comp_functional}, under the volume constraint \eqref{def_comp_constraint}. Moreover, since $a_K(z)$ is differentiable and the constraint is of isoperimetric type, there exists a Lagrange multiplier $\lambda\in \R$ such that \eqref{EL_comparison} holds.
\end{proof}

The following remark shows that the minimizer of the comparison problem reduces locally the Euclidean perimeter.

\begin{remark}[Energy estimate]
\label{rem_energy}
By minimality of~$v$ for the comparison problem~\eqref{def_comp_functional} with the volume constraint~\eqref{def_comp_constraint}, we have, via choice of $a_K\geq a$ and $\|D u\|_{L^\infty}\le K$
\[
 \int_{B_R} a(Dv) \dx'  \leq  \int_{B_R} a_K(Dv) \dx' 
   \leq  \int_{B_R} a_K(Du) \dx' =  \int_{B_R} a(Du) \dx'.
\]
Therefore, in what follows, all estimates can be written in terms of the original minimizer~$u$ only. For example, we have for every $p \in [1,2]$ that $|z|^p\le c(K)a_K(z)$, and thus
\begin{equation*}
 \int_{B_R} |Dv|^p \dx \leq c(K,n) R^{n-1}.
\end{equation*}
\end{remark}

We now derive two elementary estimates for weak solutions to the Euler--Lagrange equations
\[
\Div D_z a_K(Dv)  = \lambda f(x',v)\quad \mbox{in }B_R,\quad v=u\quad \mbox{on }\partial B_R
\]
cf.~\eqref{EL_comparison}. Our first is on the Lagrange multiplier.

\begin{lemma}[First  bound on $\lambda$]\label{rem-lambda}
There exists a constant $ R_* = R_*(n,f)$ such that if   $R \le R_0\le  R_*$ then we have the following estimate for the Lagrange multiplier $\lambda$: 
	\begin{equation*}
		\kabs{\lambda} \lesssim R^{-1}.
	\end{equation*}
	\end{lemma}
	In what follows, we will tacitly  assume that $R_0\le R_*$, so that the statement of Lemma~\ref{rem-lambda} is always true.
	
	\begin{proof}
This bound could be obtained easily by testing the weak formulation of the Euler--Lagrange equation \eqref{EL_comparison} with a suitable cut-off function, \emph{if} we knew already that the minimizer is bounded uniformly in $R\le R_0$ in the sense that $\|v\|_{L^\infty(B_R)}\le T$ for some $T$, see the discussion after \eqref{local-bound}. Because we know this to be true for the original minimizer $u$ only, we have to control their difference: Using the H\"older regularity of $f$, Jensen's inequality and the Poincar\'e inequality {(for which the Poincar\'e constant scales with $R^{2}$ on the ball $B_R$)}, we observe that
\begin{align*}
\int_{B_R} |f(x',v)-f(x',u)|\, dx' & \le [f]_{C^{0,\gamma}} \int_{B_R} |v-u|^{\gamma}\, dx'\\
& \le c(n) [f]_{C^{0,\gamma}}R^{(n-1)\frac{2-\gamma}2} \left(\int_{B_R} (v-u)^2\, dx'\right)^{\gamma/2}\\
&\le c(n,[f]_{C^{0,\gamma}}) R^{(n-1)\frac{2-\gamma}2+\gamma} \left(\int_{B_R} |Du-Dv|^2\, dx'\right)^{\gamma/2}.
\end{align*}
Now, thanks to Remark \ref{rem_energy} and the Lipschitz bound \eqref{8}, the latter leads to
\begin{equation}\label{9}
\int_{B_R} |f(x',v)-f(x',u)|\, dx' \le c(n,[f]_{C^{0,\gamma}},K) R^{n-1 + \gamma}.
\end{equation}

We finally consider a cut-off function $\eta \in C_0^\infty(B_{3R/4},[0,1])$ of radial structure satisfying $\eta \equiv 1$ in $B_{R/2}$ and $\kabs{D\eta} \lesssim R^{-1}$. In this way, we find
\begin{align*}
	\kabs{\lambda} & \leq c(n,M) R^{1-n} \Babs{ \lambda \int_{B_{\frac{3R}4}} f(x',u) \eta \dx' } \\
	& \le   c(n,M) R^{1-n} \Babs{ \int_{B_{\frac{3R}4}} D_z a_K(Dv) \cdot D \eta \dx' } + c(n,[f]_{C^{0,\gamma}},K) |\lambda| R^{ \gamma} ,
	\end{align*}
	by the virtue of \eqref{local-bound}, the Euler--Lagrange equations \eqref{EL_comparison} and the previous error estimate. Hence if $R\le R_0$ is sufficiently small so that $c(n,[f]_{C^{0,\gamma}},K)R_0^{\gamma}\le 1/2$ and using $|D_z a_K(z)|\lesssim |z|$ together with Remark \ref{rem_energy}, we conclude that
	\begin{align}
\kabs{\lambda}  &  \lesssim R^{1-n} \Babs{ \int_{B_{R}} D_z a_K(Dv) \cdot D \eta \,dx'} \label{eqn_lambda_prelim} \\
  & \lesssim  R^{1-n}\int_{B_{R}}|D_z a_K(Dv)|R^{-1}\,dx'\lesssim R^{-1}, \notag
\end{align}
which is the desired estimate.
	\end{proof}

The following lemma allows to improve the bound on the Lagrange multiplier from the previous lemma if additional error estimates are available.

\begin{lemma}[Improved bound on $\lambda$]\label{L1}
Suppose that $u$ is $C^{1,\sigma}(B_R)$ and that
\begin{equation}
\label{6}
\int_{B_R}|Du - Dv|^2 \dx \lesssim R^{n-1 +2 \theta}
\end{equation}
for some $\sigma$ and $\theta\in[0,1)$, then
\begin{equation}
\label{7} 
|\lambda | \lesssim R^{\theta-1} + R^{\sigma-1}.
\end{equation}
\end{lemma}

\begin{proof}
We take a smooth cut-off function $\eta \in C_0^\infty(B_{R},[0,1])$ satisfying $\eta \equiv 1$ in $B_{R/2}$ and $\|D \eta\|_{L^{\infty}}\lesssim R^{-1}$. We may choose this function as a test function in the Euler--Lagrange equation \eqref{EL_comparison} and find, analogously to~\eqref{eqn_lambda_prelim} in the previous proof, that 
 \begin{align*}
 |\lambda| R^{n-1} & \lesssim  \left|\int_{B_R} D_za_K(D v) \cdot D \eta\dx'\right|.
 \end{align*}
We now use the fact that $\eta$ is compactly supported {in $B_R$} to observe that $\int_{B_R}\xi \cdot D \eta\dx=0$ for any $\xi\in\R^{n-1}$. In particular, for $\xi = D_za_K((D u)_R) $, we find that
  \begin{align*}
 |\lambda| R^{n-1} &\lesssim \left| \int_{B_R}\left(D_za_K(D v) - D_za_K((D u)_R) \right)\cdot D \eta\dx'\right|\\
 & \lesssim R^{-1} \int_{B_R} \left|D_za_K(D v)  - D_za_K((D u)_R)\right|\dx'.
 \end{align*}
It is easily seen that $D_za_K$ is Lipschitz with a Lipschitz constant depending just on $K$. Using in addition the triangle inequality, we thus obtain
\[
 |\lambda| R^{n} \lesssim \int_{B_R} |D v - D u|\dx' + \int_{B_R} |D u - (D u)_R|\dx'.
 \]
 For the first term, we use  Jensen's inequality and the hypothesis \eqref{6} to bound
 \[
 \int_{B_R} |D v - D u|\dx  \lesssim R^{n-1+\theta}.
 \]
 For the second term, we use the fact that $D u$ is known to be a $\sigma$-H\"older function, and thus
 \[
 \int_{B_R} |D u - (D u)_R|\dx\lesssim R^{n-1+\sigma}.
 \]
 A combination of the previous bounds yields the statement of the lemma.
\end{proof}

Our next result is somehow complementary to the previous lemma.

\begin{lemma}[First error estimate]\label{lambda-1}
Suppose that there exists $\delta \in [0,1)$ such that
$$|\lambda|\lesssim R^{\delta - 1}.$$
Then, 
\begin{equation}\label{u-v}
	\int_{B_R} \kabs{Du - Dv}^2 \dx' \lesssim R^{ n-1 +\frac{2\alpha}{2-\alpha}} + R^{ n-1 +{ \frac{2}{1-\gamma}(\gamma +\delta)}}. 
\end{equation} 
\end{lemma}

\begin{proof}
Making use of the strong convexity of $a_K$, cf.~\eqref{12}, and the Euler--Lagrange equation~\eqref{EL_comparison} for $v$ (tested with $\varphi = u - v$), we find
	\begin{equation}\label{1}
		\frac{\mu}2\int_{B_R} \kabs{Du - Dv}^2 \dx' \le  \int_{B_R} \big( a_K(Du) - a_K(Dv) \big) \dx' + \lambda  \int_{B_R}  f(x',v(x')) (u(x')-v(x')) \dx' .
	\end{equation}
	Now we estimate the two integrals appearing on the right-hand side separately. We start by noticing that the first integral is non-negative due to the fact that $v$ is a minimizer. Hence, using in addition the Lipschitz bound on $u$, the definition of $a_K$ and the lower bound on $h$, we observe that
	\begin{align*}
		\frac1M \int_{B_R} \big( a_K(Du) - a_K(Dv) \big) \dx'& \le h(0',u(0')) \int_{B_R} \left(a(Du)-a(Dv)\right)\dx'\\
		& = \int_{B_R} \big( h(0', u(0')) - h(x',u(x')) \big) \big( a(Du)  - a(Dv) \big) \dx' \\
		& \quad {}+ \int_{B_R} \big( h(x', u(x')) a(Du) - h(x',v(x')) a(Dv) \big) \dx' \\
		& \quad {}+ \int_{B_R} \big(  h(x',v(x')) - h(x', u(x')) \big) a(Dv) \dx'.
	\end{align*}
	The second term on the right-hand side is non-positive because $u$ is a minimizer for the initial full weighted problem and  $v$ is an admissible competitor.
	Using now that $h$ is H\"older of order~$\alpha$, $u$ is Lipschitz with constant $K$ and $a$ is Lipschitz with constant $1$, we can further estimate
	\begin{align*}
		\mel \frac1M \int_{B_R} \big( a_K(Du) - a_K(Dv) \big) \dx'\\
		& \leq \left(1+K^{\alpha}\right) [h]_{C^{0,\alpha}} \int_{B_R}  \kabs{x'}^\alpha  \kabs{Du - Dv} \dx' 
		+ [h]_{C^{0,\alpha}} \int_{B_R} |v-u|^{\alpha} a(Dv)\dx' .
	\end{align*}
	We estimate with the help of Young's inequality:
	\begin{align*}
		\mel
		\frac{1}{M}\int_{B_R} \big( a_K(Du) - a_K(Dv) \big) \dx' \\
		& \le \eps \int_{B_R} \big( \kabs{Du - Dv}^2 + R^{-2} \kabs{u-v}^2 \big) \dx'  + C \left( \int_{B_R} |x'|^{2\alpha}\dx'+ R^{\frac{2\alpha}{2-\alpha}}\int_{B_R} a(Dv)^{\frac{2}{2-\alpha}}\dx'\right),
	\end{align*}
	where $\eps$ is some positive and small but finite constant that we will fix later and where $C=C(\eps,K,\alpha,M,[h]_{C^{0,\alpha}},n)$ is a constant that may (from here on) change from line to line. Clearly, the second term on the right-hand side is of the order $R^{n-1+2\alpha}$. Moreover, because $a(z)^p\le C(K) a_K(z)$ for any $p\in[1,2]$ and thanks to the minimality of~$v$ and the Lipschitz bound on $u$, we have for the third term that
	\[
	\int_{B_R} a(Dv)^{\frac{2}{2-\alpha}}\dx' \lesssim \int_{B_R} a_{K}(Dv)\dx'\le \int_{B_R} a_{K}(Du)\dx' \le C R^{n-1}.
	\]
	For $R\le 1$, we conclude that
	\[
	\frac{1}{M}\int_{B_R} \big( a_K(Du) - a_K(Dv) \big) \dx' \le \eps \int_{B_R} \big( \kabs{Du - Dv}^2 + R^{-2} \kabs{u-v}^2 \big) \dx' + CR^{n-1+\frac{2\alpha}{2-\alpha}}.
	\]
	
	It remains to estimate the volume constraint term in \eqref{1}. Using the assumption on the Lagrange multiplier, and the fact that $u$ and $v$ satisfy the same volume constraint, we  estimate
	\begin{align*}
		\mel | \lambda|  \Babs{ \int_{B_R}  f(x',v(x')) (u(x')-v(x')) \dx'}\\
		&\lesssim R^{\delta-1} \Babs{\int_{B_R} \int_0^{u(x')} f(x',v(x'))\dt\dx' - \int_{B_R}\int_0^{v(x')} f(x',v(x'))\dt\dx'}\\
		& = R^{\delta-1} \Babs{\int_{B_R} \int_{v(x')}^{u(x')} \left(f(x',v(x')) - f(x',t)\right)\dt\dx'}.
	\end{align*}
	We use the $\gamma$-H\"older regularity of $f$ and Young's inequality to get
	\begin{align*}
		| \lambda|  \Babs{ \int_{B_R}  f(x',v(x')) (u(x')-v(x')) \dx'} & \le C R^{\delta-1} \int_{B_R} \kabs{u(x') - v(x')}^{1 + \gamma} \dx' \\
		& \leq \eps  \int_{B_R} R^{-2} \kabs{u-v}^2 \dx' + C R^{ n-1 + { \frac{2}{1-\gamma}(\gamma +\delta)}}.
	\end{align*} 
	
	We have now a bound on both terms appearing on the right-hand side of \eqref{1}. Therefore,
	\[
	\frac{\mu}2 \int_{B_R} \kabs{Du - Dv}^2 \dx' \le 2\eps \int_{B_R} \big( \kabs{Du - Dv}^2 + R^{-2} \kabs{u-v}^2 \big) \dx' + CR^{n-1+\frac{2\alpha}{2-\alpha}}+C R^{ n-1+ { \frac{2}{1-\gamma}(\gamma +\delta)}}.
	\]
	Hence, via Poincar\'e's inequality and a suitable (small) choice of $\eps$, we end up with \eqref{u-v}.
\end{proof}

Iterating Lemma \ref{lambda-1} and \ref{L1}, we get the following

\begin{corollary}[Improved error estimate]\label{Du-Dv}
Let $u\in C^{1,\sigma}(B_R)$ with $\sigma \leq \tfrac{\alpha}{2-\alpha}$, then we have that
\begin{equation}\label{lambda}|\lambda|\lesssim R^{\sigma-1}\end{equation}
and 
\begin{equation}\label{opt}\int_{B_R}|Du-Dv|^2\,dx'\lesssim R^{n-1+\frac{2\alpha}{2-\alpha}}+ R^{n-1+ { \frac{2}{1-\gamma}(\gamma+\sigma)}}.
	\end{equation}
\end{corollary} 
\begin{proof}
	By Lemma~\ref{rem-lambda} we know that $|\lambda|\lesssim R^{-1}$, hence we can apply Lemma \ref{lambda-1} with $\delta=0$ and deduce that
	$$\int_{B_R}|Du-Dv|^2\lesssim R^{n-1+\frac{2\alpha}{2-\alpha}}+ R^{n-1+ \frac{2\gamma}{1-\gamma}}.$$
	Under the assumption \eqref{2}, the second term on the right-hand side is the leading order term for $R\le 1$.	Hence, applying Lemma \ref{L1} (with $\theta=\frac{\gamma}{1-\gamma}$), we deduce that 
		$$|\lambda|\lesssim R^{\sigma-1} + R^{\frac{\gamma}{1-\gamma}-1}.$$
		If $\frac{\gamma}{1-\gamma}\ge \sigma$, the first bound \eqref{lambda} is proved and \eqref{opt} follows directly from Lemma \ref{lambda-1}.
		Otherwise, we can iterate the above procedure and after a finite number of steps we will reach the estimate 
		$$|\lambda|\lesssim R^{\sigma-1},$$ which, again, will imply, by using Lemma \ref{lambda-1}, the desired estimate \eqref{opt}.
		\end{proof}
	


%

%
%

%
%
%

\section{Decay estimates for the comparison problem}

In this section, we establish decay estimates for the solution $v$ of the comparison problem. In getting such estimates, the bound on the Lagrange multiplier $\lambda$ obtained in the previous section will play a crucial role.

We start, however, with first (suboptimal) H\"older estimates for the minimizer of the comparison problem.

\begin{lemma}\label{regu-v}
	Let $v \in u + W^{1,2}_0(B_R)$ be a minimizer for the functional~\eqref{def_comp_functional} under the constraint~\eqref{def_comp_constraint} in the above setting.
	Then $v \in W^{2,q}_{\text{loc}}(B_R)$ for any $q\in(1,\infty)$, and thus, in particular, it holds $v \in C^{1,\omega}(B_{\frac{R}2})$ for any $\omega \in (0,1)$. Moreover, there are the estimates
	\[
	\|Dv\|_{L^{\infty}(B_{\frac{R}2})} \lesssim1 \quad \mbox{and}\quad [Dv]_{C^{0,\omega}(B_{\frac{R}2})} \lesssim R^{-\omega}.
	\]
	\end{lemma}
	
	The proof of the statement follows from standard elliptic theory. We provide it for the convenience of the reader.
	
	\begin{proof}
	We start by rewriting the Euler--Lagrange equation as
	\[
	-D_z^2 a_K(D v):D^2v + \lambda f(x',v)=0\quad \mbox{in }B_R,
	\]
	and we recall that $A \coloneqq D_z^2 a_K(D v)$ is a uniformly elliptic matrix by the virtue of \eqref{13}. By rescaling
	\[
	x'=R\hat x', \quad v(x')=R\hat v(\hat x'),\quad \lambda = R^{-1} \hat \lambda,\quad f(x',v) = \hat f(\hat x',\hat v), \quad A(x') = \hat A(\hat x'),
	\]
	observing that $[\hat f]_{C^{0,\gamma}}   =R^{\gamma} [f]_{C^{0,\gamma}} \lesssim 1$  because $R\le1$, and invoking Remark~\ref{rem_energy} and {Lemma}~\ref{rem-lambda}, it is enough to consider the case $R=1$.
	
	We introduce a cut-off function $\eta$ whose  support is compactly contained in $B_1$. Smuggling this function into the elliptic equation leads to considering 
	\begin{equation}
	\label{5}
	-A:D^2 w +  w  = - 2A:D\eta\otimes Dv - A:(D^2\eta) v + \eta v - \lambda \eta f(\cdot,v)\quad \mbox{in }\R^n.
	\end{equation}
	By the assumptions of the lemma and recalling that $f$ is H\"older continuous,  arguing similarly as in the proof of Lemma \ref{rem-lambda}, we see that the right-hand side belongs to $L^2(\R^n)$, and thus, by standard theory for elliptic equations, e.g.,~Theorem 5.1.1 in \cite{Krylov_Sob}, there exists a unique solution $w$, which must coincide with $\eta v$ by construction, and that solution belongs to $W^{2,2}(\R^n)$. Moreover, thanks to Remark~\ref{rem_energy} and {Lemma}~\ref{rem-lambda} and Poincar\'e's inequality,  we have the estimate
	\[
	\|D^2 w\|_{L^2} + \|w\|_{L^2} \lesssim 1,
	\]
	by inspection of the right-hand side.
	Invoking the Sobolev embedding theorem and recalling that $w$ is compactly supported, we deduce that $w\in W^{1,q}(\R^n)$ for any $q\in [1,\frac{2n}{n-2})$ for $n \geq 3$ and any $q \in [1,\infty)$ for $n=2$, and thus, the right-hand side of \eqref{5} must be in $L^q(\R^n)$ with norm estimate
	\[
	\|D w\|_{L^q} +  \|w\|_{L^q} \lesssim 1.
	\]
	This procedure can be repeated and we deduce that $w\in W^{2,q}(\R^n)$ for any $q\in(1,\infty)$ eventually. We finally make use of the Sobolev embedding into H\"older spaces to conclude that $w\in C^{1,{\omega}}(\R^n)$ for any ${ \omega}\in(0,1)$. Choosing $\eta$ appropriately gives the  statement of the lemma.
	\end{proof}

The following Proposition is the main result of this Section: it establishes decay estimates for the oscillation of $\partial_i v$, being $v$ the solution of the comparison problem introduced in the previous Section.
\begin{proposition}\label{decay-comparison}
	Let $u\in C^{1,\sigma}(B_R)$ with $\sigma \leq \tfrac{\alpha}{2-\alpha}$ and let $v \in u + W^{1,2}_0(B_R)$ be a solution of \eqref{EL_comparison}, under the volume constraint  \eqref{def_comp_constraint}. Then there exists $ \rho_0>0$ of the form $ \rho_0=\eps_0 R$ \textup{(}with $\eps_0$ depending only on $n,\,K,\,M,\,\alpha,\,\gamma, \, [f]_{C^{0,\gamma}}$ and $[h]_{C^{0,\alpha}}$\textup{)} such that $B_{\rho_0}(x_0') \subset B_{R/4}$, and for all $0<r<\rho \le \rho_0$, we have
	\begin{equation}\label{decay-v}
		\int_{B_r(x_0')}|\partial_i v -(\partial_i v)_r|^2\dx' \lesssim \left(\frac{r}{\rho}\right)^{{n-1+2(\gamma+\sigma)}}\int_{B_\rho(x_0')}|\partial_i v -(\partial_i v)_\rho|^2\dx' + { r^{n-1+2(\gamma+\sigma)}},
		\end{equation}
	for any $i=1,\dots,n-1$.
\end{proposition}

In order to prove the previous result, we will consider the equation satisfied by the derivatives of $v$ (in the weak sense): 
\begin{equation}\label{eq-dv}
	\Div \big( D^2_z a_K(Dv) D \partial_i v \big) = \lambda \partial_i \left( f(x',v)\right)
	\end{equation}
	for every $i \in \{1,\ldots,n-1\}$. Thus, each of the functions $\partial_i v$ solves an equation with measurable, elliptic (cf.~\eqref{13}) and bounded coefficients (given by $D^2_z a_K(Dv)$), where the inhomogeneity is the ``derivative'' of a function in $L^2(B_R)$. 
	In order to get the desired decay estimates for $\partial_i v$, we need some intermediate decay estimates for the solutions to an associate problem with constant coefficients.

	Let $\rho>0$ be such that $B_\rho(x_0')\subset B_{R/4}$ and let $A(Dv)$ denote the matrix $D^2_z a_K(Dv)$ and $A_0$ the constant matrix obtained by freezing the coefficients in $B_{\rho}(x_0')$, more precisely, $A_0=A((Dv)_{B_\rho(x_0')})$. Moreover, set $f_0=f(x_0',(v)_{B_\rho(x_0')})$. With these notations, we have the following result.
	\begin{lemma}\label{decay-w}
	Let $u\in C^{1,\sigma}(B_R)$ with $\sigma \leq \tfrac{\alpha}{2-\alpha}$ and let $w \in \partial_i v + W^{1,2}_0(B_\rho(x_0'))$ be a weak solution of the linear elliptic Dirichlet problem 
	\begin{equation*}
		-\Div (A_0 Dw)= - \lambda \partial_i (f(x',v)-f_0)\quad \mbox{in } B_\rho(x_0').
		\end{equation*}	
	Then, for any $0<r<\rho$ we have:
	\begin{equation}\label{decay-est-w}
		\int_{B_r(x_0')}|Dw|^2\dx' \lesssim \left(\frac{r}{\rho}\right)^{n-1}\int_{B_\rho(x_0')} |Dw|^2 \dx'+ R^{2(\sigma -1)}\rho^{n-1+{ 2\gamma}}.
		\end{equation}
		\end{lemma}
\begin{proof}
	
For ease of notation, we do not write explicitly the center of the balls. Hence in the following computations $B_\rho$ and $B_r$ stand for $B_\rho(x_0')$ and $B_r(x_0')$, respectively. 
	
	We write $w=\psi + \phi$ where $\psi$ is the weak solution of the corresponding homogeneous problem
	\begin{equation*}
		\begin{cases}
			-\text{div}(A_0 D \psi)=0 & \mbox{in } B_\rho,\\
			\psi =\partial_i v & \mbox{on } \partial B_\rho,
		\end{cases}
	\end{equation*}
and $\phi$ satisfies the {inhomogeneous} problem
		\begin{equation}\label{problem-phi}
		\begin{cases}
			-\text{div}(A_0 D \phi)=\lambda \partial_i (f(x',v)-f_0)& \mbox{in } B_\rho,\\
			\phi =0 & \mbox{on } \partial B_\rho.
		\end{cases}
	\end{equation}
By standard decay estimates for the homogeneous equation with constant coefficients (see, e.g., Lemma 2.17 in \cite{ACM}), we have 
\begin{equation}\label{hom}\int_{B_r}|D\psi|^2 \, dx'\lesssim \left(  \frac{r}{\rho}\right)^{n-1}\int_{B_{\rho}} |D \psi|^2\, dx',
	\end{equation}
	for any $0<r<\rho$.
	
	We consider now the {inhomogeneous} problem satisfied by $\phi$ and test it with $\phi$ itself, to get in a first step by employing the ellipticity of~$A_0$ in~\eqref{13}, integrating by parts and using Young's inequality
	\begin{equation*}
	 \mu \int_{B_{\rho}} |D\phi|^2 \dx' \leq \frac{\lambda^2}{2\mu} \int_{B_\rho} |f(x',v)-f(x_0',(v)_\rho)|^2 \dx' + \frac{\mu}{2} \int_{B_{\rho}}|D\phi|^2\dx'.
	\end{equation*}
	Then, after absorbing the second integral on the right-hand side in the left-hand side, we find
	\begin{equation}\label{nonhom}
	\begin{aligned}
			\int_{B_{\rho}}|D\phi|^2\dx' & \lesssim \lambda^2 \int_{B_\rho} |f(x',v)-f(x_0',(v)_\rho)|^2 \dx'\\
			&\lesssim  R^{2(\sigma-1)}\int_{B_{\rho}} \left(|x'-x_0'|^{2\gamma} + |v-(v)_{\rho}|^{2\gamma}\right)\, dx'\\
			&\lesssim R^{2(\sigma-1)} \rho^{n-1+2\gamma},
		\end{aligned}
		\end{equation}
	where we have used that $f\in C^{0,\gamma}$,  that $v$ is Lipschitz with a Lipschitz constant that does not depend on $R$, see Lemma \ref{regu-v}, and applied the bound on $\lambda$ given in Corollary \ref{Du-Dv}. 
	
	Finally, we combine \eqref{hom} and \eqref{nonhom} {(and add $\pm D\phi$)}, to get
	\begin{equation*}
		\begin{split}
			\int_{B_r}|Dw|^2\dx'& \lesssim \int_{B_r}|D\psi|^2\dx' + \int_{B_r}|D\phi|^2\dx'\\
			&\lesssim\left(\frac{r}{\rho} \right)^{n-1} \int_{B_\rho}|D\psi|^2   \dx'+ {\ R^{2(\sigma-1)}\rho^{n-1+2\gamma}}\\
			&\lesssim \left(\frac{r}{\rho} \right)^{n-1}\int_{B_\rho}|Dw|^2 \dx'+ \left(\frac{r}{\rho} \right)^{n-1}\int_{B_\rho}|D\phi|^2 \dx'+ { R^{2(\sigma-1)}\rho^{n-1+2\gamma}} \\
			&\lesssim  \left(\frac{r}{\rho} \right)^{n-1}\int_{B_\rho}|Dw|^2\dx' + { R^{2(\sigma-1)}\rho^{n-1+2\gamma}}.
		\end{split}
	\end{equation*}
	This is the stated estimate.
\end{proof}

We can now give the proof  of the decay estimates for our comparison problem.

\begin{proof}[Proof of Proposition \ref{decay-comparison}]
Let $\rho$ be such that $B_\rho(x_0')\subset B_{R/4}$. We shall again neglect the actual center of the ball for notational convenience, that is, $B_{r} = B_{r}(x_0')$ for any $r\le \rho$ in the following. Let $w$ be as in Lemma \ref{decay-w}, then we have that the function $\partial_ i v - w$ is a weak solution of the following problem: 
	\begin{equation*}
		\begin{cases}
			-\Div(A_0D(\partial_i v - w))=-\Div\left((A_0-A(Dv))D\partial_i v\right) &\mbox{in } B_\rho,\\
			\partial_ i v - w=0 & \mbox{on }\partial B_\rho.
		\end{cases}
		\end{equation*}

	We now test the above equation with $\partial_i v - w$ itself, to get {first via the ellipticity of $A_0$ from \eqref{13}
	\begin{equation*}
	 \mu \int_{B_\rho }|D(\partial_i v - w)|^2 \dx' \leq \frac{1}{2\mu} \int_{B_{\rho} }|A_0-A(Dv)|^2 |D\partial_i v|^2\dx' + \frac{\mu}{2} \int_{B_\rho }|D(\partial_i v - w)|^2 \dx' 
	\end{equation*}
    and then}
	\begin{equation}\label{v-w}
		\int_{B_\rho }|D(\partial_i v - w)|^2 \dx' \lesssim \int_{B_{\rho} }|A_0-A(Dv)|^2 |D\partial_i v|^2\dx' \lesssim \left(\frac{\rho}{R}\right)^{2\omega}\int_{B_{\rho} } |D\partial_i v|^2\dx',
	\end{equation}
where we have used that $A$ is Lipschitz and $Dv$ is $\omega$-H\"older continuous with $[Dv]_{C^{0,\omega}}\lesssim R^{-\omega}$ by Lemma \ref{regu-v}.
Thus, we have 
\begin{equation*}
	\begin{split}
		\int_{B_r}|D\partial_i v|^2\dx' &\lesssim \int_{B_r} |Dw|^2 \dx'+ \int_{B_r}|D(\partial_i v - w)|^2\dx' \\
		&\lesssim \left(\frac{r}{\rho}\right)^{n-1} \int_{B_\rho} |Dw|^2\dx' + \int_{B_\rho}|D(\partial_i v - w)|^2 \dx'+ {R^{2(\sigma-1)}\rho^{n-1+2\gamma}}\\
		& \lesssim \left(\frac{r}{\rho}\right)^{n-1} \int_{B_\rho} |D\partial_i v|^2\dx' + \int_{B_\rho}|D(\partial_i v - w)|^2 \dx'+ {R^{2(\sigma-1)}\rho^{n-1+2\gamma}}\\
		& \lesssim\left( \left(\frac{r}{\rho}\right)^{n-1}+\left(\frac{\rho}{R}\right)^{2\omega}\right) \int_{B_\rho} |D\partial_i v|^2 \dx'+ {R^{2(\sigma-1)}\rho^{n-1+2\gamma}},
	\end{split}
\end{equation*}
where we have used Lemma~\ref{decay-w} for the second inequality{, added $\pm D \partial_i v$ for the third,} and exploited estimate~\eqref{v-w} for the last one.
Applying now the Poincar\'e inequality on the left-hand side, we deduce
\begin{equation}
	\begin{split}
	\int_{B_r}|\partial_i v- (\partial_i v)_r|^2\dx'& \lesssim r^2\int_{B_r}|D\partial_i v|^2\dx' \\
		&\lesssim r^2\left( \left(\frac{r}{\rho}\right)^{n-1}+\left(\frac{\rho}{R}\right)^{2\omega}\right) \int_{B_\rho} |D\partial_i v|^2\dx' + { r^2 R^{2(\sigma-1)}\rho^{n-1+2\gamma}}\\
	  &\lesssim r^2\left( \left(\frac{r}{\rho}\right)^{n-1}+\left(\frac{\rho}{R}\right)^{2\omega}\right) \int_{B_\rho} |D\partial_i v|^2 + {  \left(\frac{\rho}{R}\right)^{2(1-\sigma)}\rho^{n-1+2(\gamma+\sigma)}}.	\label{osc-v}
	\end{split}
\end{equation}
We now use a Caccioppoli-type estimate for the equation satisfied by $\partial_i v$ in order to replace the quantity $|D \partial_i v|^2$ on the right-hand side of the previous inequality by the oscillation $|\partial_i v -(\partial_i v)_\rho|^2$. We derive such a Caccioppoli estimate in a standard way, by testing the equation 
$$-\text{div}(A(Dv)D\partial_i v)= - \lambda \partial_i (f-f_0)$$
with $\eta^2(\partial_i v - (\partial_i v)_{2\rho})$, where $\eta \in C_0^\infty(B_{2\rho},[0,1])$ is a standard cut-off function satisfying $\eta \equiv 1$ in $B_{\rho}$ and $\|D \eta\|_{L^{\infty}}\lesssim \rho^{-1}$. 
In this way we obtain { first
\begin{multline*}
 \mu \int_{B_{2\rho}}\eta^2 |D \partial_i v|^2 \dx' \\ 
  \leq 4 \int_{B_{2\rho}} |A(Dv)|^2 |D\eta|^2 |\partial_i v - (\partial_i v)_{2\rho}|^2  \dx' + \frac{\lambda^2 }{\mu} \int_{B_{2\rho}}|f-f_0|^2\dx' + \frac{\mu}{2} \int_{B_{2\rho}}\eta^2 |D \partial_i v|^2 \dx'
\end{multline*}
and then}
\begin{equation}\label{Dv}
	\begin{aligned}
	\int_{B_\rho }|D \partial_ i v|^2 \dx' &\leq \int_{B_{2\rho}}\eta^2 |D \partial_i v|^2\dx'\\
	& \lesssim \int_{B_{2\rho}}|D\eta|^2 |\partial_i v - (\partial_i v)_{2\rho}|^2 \dx'+ \lambda^2 \int_{B_{2\rho}}|f-f_0|^2\dx'\\
	& \lesssim \rho^{-2} \int_{B_{2\rho}} |\partial_i v - (\partial_i v)_{2\rho}|^2 \dx'+ { R^{2(\sigma-1)}\rho^{n-1+2\gamma}},
	\end{aligned}
\end{equation}
where we have used an estimate that is almost identical to \eqref{nonhom} to control  the inhomogeneity. (Here, we need that $B_\rho \subset B_{R/4}$.)
Plugging \eqref{Dv} into \eqref{osc-v}, we conclude the following decay estimate for the oscillation of $\partial_ i v$:
\begin{align*}
\mel\int_{B_r}|\partial_i v-(\partial_i v)_r|^2\dx'\\
& \lesssim \left(\frac{r}{\rho} \right)^{2} \left( \left(\frac{r}{\rho}\right)^{n-1} +  \left(\frac{\rho}{R}\right)^{2\omega}\right) \int_{B_{2\rho}}|\partial_i v - (\partial_i v)_{2\rho}|^2\dx' + { \left(\frac{\rho}{R}\right)^{2(1-\sigma)}\rho^{n-1+2(\gamma+\sigma)}}\\
&\lesssim  \left( \left(\frac{r}{\rho}\right)^{n+1} +  \left(\frac{\rho}{R}\right)^{2\omega}\right) \int_{B_{2\rho}}|\partial_i v - (\partial_i v)_{2\rho}|^2 \dx'+ { \rho^{n-1+2(\gamma+\sigma)}},
\end{align*}
for any $0<r<\rho$ with $B_\rho(x_0')\subset  B_{R/4}$. Observe that the previous estimate extends trivially to $r \in (\rho, 2 \rho)$ thanks to the monotonicity of the oscillation discussed right after Proposition~\ref{lemma_Campanato}. Hence, the estimate is valid  for any $r<2\rho$.

For $\varepsilon >0$ given as in Lemma~\ref{lemma_iteration}, let now $ \rho_0$ be such that  $B_{\rho_0}(x_0')\subset B_{R/4}$ and satisfying
$$\left(\frac{ \rho_0}{R}\right)^{2\omega}\le \varepsilon .
$$
It follows that 
\begin{equation}
		\int_{B_r}|\partial_i v-(\partial_i v)_r|^2 \dx' \lesssim
		  \left( \left(\frac{r}{\rho}\right)^{n+1} +  \varepsilon\right) \int_{B_{\rho}}|\partial_i v - (\partial_i v)_{\rho}|^2\dx' + {  \rho^{n-1+2(\gamma+\sigma)}},
\end{equation}
holds for any $r<\rho\le \rho_0$.
Finally, we can apply the iteration Lemma~\ref{lemma_iteration}, to deduce the desired decay estimate for $\partial_i v$ and conclude the proof of the proposition. 
	\end{proof}

\section{Proof of the regularity result}
\label{Section_regularity_proof}
We are now ready to prove our main result. The idea consists in transferring the oscillation decay estimate from Proposition~\ref{decay-comparison} for the comparison function $v$ to our minimizer $u$, by making use of the error estimate from  Lemma \ref{lambda-1}.

\begin{proof}[Proof of Theorem \ref{thm_main}]
As already discussed in the introduction, the optimal bound on the dimension of the singular set follows by the classical regularity theory for $\omega$-minimal sets established in \cite{Tam}. Indeed, in \cite{CP1,PS}, it was proved that, under our assumption, any isoperimetric set is $\omega$-minimal with $\omega(r)=r^{2\sigma}$, being $\sigma=\alpha/(2n(1-\alpha)+2\alpha)$. 
It remains, thus, to show the optimal regularity of the reduced boundary.
		
Let $u$ be the local representation of $\partial^* E$, i.e., the solution of our weighted minimization problem \eqref{3} under the weighted volume constraint \eqref{4}, which initially enjoys the regularity $u \in C^{1,\sigma}(B_R)$ with $\sigma=\alpha/(2n(1-\alpha) + 2\alpha)$ according to Theorem~\ref{prop_densities_low_reg2}, and let $v$ be the comparison function studied in the previous section.
Using {the minimality~\eqref{eq_minimality_mean_value} of mean values,} Proposition \ref{decay-comparison}, and Corollary \ref{Du-Dv}, we deduce that 
\begin{equation}\label{transfer}
	\begin{aligned}
	\mel	\int_{B_r}|\partial_i u -(\partial_i u)_r|^2 \dx'\\
	& {\leq \int_{B_r}|\partial_i u -(\partial_i v)_r|^2 \dx'}\\
		& \lesssim \int_{B_r}|\partial_i v -(\partial_i v)_r|^2\dx' + \int_{B_R}|Du - Dv|^2\dx'\\
	&\lesssim  \left(\frac{r}{\rho} \right)^{n-1+2(\gamma+\sigma)}\int_{B_\rho}|\partial_i v -(\partial_i v)_\rho|^2\dx' + r^{n-1+2(\gamma+\sigma)}+R^{n-1+\frac{2\alpha}{2-\alpha}} + R^{n-1+\frac{2}{1-\gamma}(\gamma+\sigma)},
	\end{aligned}
	\end{equation}
 for any $0<r<\rho\le \rho_0=\eps_0 R$. 
 

 Employing once again the minimality~\eqref{eq_minimality_mean_value} of mean values combined with Corollary \ref{Du-Dv}, we can pass on the right-hand side to $\partial_i u$ instead of $\partial_i v$, and with $r\le R\le 1$, this implies the following oscillation decay for $\partial_i u$:
\begin{equation}
	\int_{B_r}|\partial_i u -(\partial_i u)_r|^2 \dx'\lesssim \left(\frac{r}{\rho} \right)^{n-1+2(\gamma+\sigma)}\int_{B_\rho}|\partial_i u -(\partial_i u)_\rho|^2 \dx' + R^{n-1+2\min\{\gamma+\sigma,\frac{\alpha}{2-\alpha}\}} ,
\end{equation} 
for any $0<r<\rho \le \rho_0$.
We choose now $\rho= \rho_0 =\eps_0 R$, 
and deduce that
\begin{equation}
	\begin{aligned}
	\mel	\int_{B_r}|\partial_i u -(\partial_i u)_r|^2\dx'\\
	 &\lesssim \left(\frac{r}{R} \right)^{n-1+2(\gamma+\sigma)}\int_{B_{\eps_0 R}}|\partial_i u -(\partial_i u)_{\eps_0 R}|^2 \dx'+ R^{n-1+2\min\{\gamma+\sigma,\frac{\alpha}{2-\alpha}\}}\\
	&\lesssim \left(\frac{r}{R} \right)^{n-1+2(\gamma+\sigma)}\int_{B_{R}}|\partial_i u -(\partial_i u)_{R}|^2 \dx'+ R^{n-1+2\min\{\omega,\frac{\alpha}{2-\alpha}\}},
		\end{aligned}
\end{equation} 
for any $0<r\le \eps_0 R$, where $\omega = \gamma+\sigma-\delta>\sigma$ for any $\delta\in(0,1)$ small. We may choose $\omega = \frac{\gamma}2+\sigma$. The same estimate trivially extends to $r \in (\eps_0 R, R)$ and thus holds for any $0<r<R$.

We can now apply again the iteration Lemma~\ref{lemma_iteration}, to get
\begin{equation*}
	\int_{B_r(x_0')}|\partial_i u - (\partial_i u)_r|^2 \dx' \lesssim \left(\frac{r}{R}\right)^{n-1+2\min\{\omega,\frac{\alpha}{2-\alpha}\}}\int_{B_R}|\partial_i u - (\partial_i u)_{{R}}|^2 \dx' + r^{n-1+2\min\{\omega,\frac{\alpha}{2-\alpha}\}}.
\end{equation*}
Finally, by Proposition~\ref{lemma_Campanato}, we deduce that $u\in {{C^{1,\min\{\frac{\alpha}{2-\alpha},\omega \} }}}$. If $\omega  = \frac{\gamma}2+\sigma \ge \frac{\alpha}{2-\alpha}$ the proof is completed. Otherwise we can iterate the above reasoning: setting $\sigma_j \coloneqq \sigma + \frac{j\,\gamma}2$, we can iteratively apply Proposition \ref{decay-comparison} and Corollary \ref{Du-Dv}, with $u\in C^{1, \sigma_j}$ and plug the new improved estimate \eqref{opt} (with $\sigma_j>\sigma$ in place of $\sigma$) into \eqref{transfer}. After a finite number $N$ of steps (in particular when $N\gamma/2 + \sigma \ge \frac{\alpha}{2-\alpha}$) we reach the exponent $\frac{\alpha}{2-\alpha}$.  This concludes the proof of Theorem \ref{thm_main}.
\end{proof}


	\section*{Acknowledgements}

EC gratefully acknowledges the kind hospitality of the Universit\"at Augsburg. Her work is partially supported by the Gruppo Nazionale per l’Analisi Matematica, la Probabilit\`a e le loro Applicazioni (GNAMPA) of the Istituto Nazionale di Alta Matematica (INdAM), and by the Spanish grant PID2021-123903NB-I00 funded by MCIN/AEI/10.13039/501100011033 and by ERDF "A way of making Europe".

CS gratefully acknowledges the kind hospitality of the Universit\`a di Bologna. His work is funded by the Deutsche Forschungsgemeinschaft (DFG, German Research Foundation) under Germany's Excellence Strategy EXC 2044--390685587, Mathematics M\"unster: Dynamics Geometry Structure.

\end{document}